\documentclass[12pt]{amsart}
\usepackage{hyperref}
\usepackage{amsmath, amsthm,amstext}
\usepackage{amssymb, array, amsfonts}
\usepackage[english]{babel}
\usepackage{url}
\usepackage{amsrefs}
\usepackage{float}
\usepackage{graphics}
\usepackage{graphicx}
   \usepackage{enumerate}
\numberwithin{equation}{section}

\usepackage[margin=1in]{geometry}
\def\dual{\mathrm{dual}}

\def \cQ {\mathcal Q}
\renewcommand{\l}{\left}
\renewcommand{\r}{\right}

\def \Q{\mathbb{Q}}

\def \M2{\mathrm{M}_2}
\def \R{\mathbb{R}}
\def \Z{\mathbb{Z}}
\def \T{\mathbb{T}}

\def \sl2r{\mathrm{SL}(2,\R)}

\newcommand{\beq}{\begin{equation}}
\newcommand{\eeq}{\end{equation}}

\def \dirint {\int_{\T}^{\oplus}}
\def \dirintd {\int_{\T^d}^{\oplus}}

\DeclareMathOperator*{\slim}{s--lim}
\def\dom{\operatorname{Dom}}

\def\wt{\widetilde}

\def\wt{\widetilde}

\def\ess{\operatorname{ess}}

\newcommand{\eqdef}{\stackrel{\rm def}{=\kern-3.6pt=}}
	\newcommand{\<}{\langle}
\renewcommand{\>}{\rangle}

\theoremstyle{plain}
\newtheorem{theorem}{\bf Theorem}[section]
\newtheorem{lemma}[theorem]{\bf Lemma}

\newtheorem{prop}[theorem]{\bf Proposition}
\newtheorem{cor}[theorem]{\bf Corollary}

\theoremstyle{definition}
\newtheorem{defi}[theorem]{\bf Definition}

\theoremstyle{remark}
\newtheorem{remark}[theorem]{\bf Remark}

\renewcommand{\le}{\leqslant}
\renewcommand{\ge}{\geqslant}

\newcommand{\dist}{\mathop{\mathrm{dist}}\nolimits}

\newcommand{\dc}{\mathop{\mathrm{DC}}\nolimits}

\title[Ballistic transport for quasiperiodic operators]{On the relation between strong ballistic transport and exponential dynamical localization}
\author[I. Kachkovskiy]{Ilya Kachkovskiy}
\address{Department of Mathematics,
	Michigan State University,
	Wells Hall, 619 Red Cedar Road,
	East Lansing, MI 48824,
	United States of America}
\email{ikachkov@msu.edu}

\date{}
\begin{document}
\begin{abstract}
We establish strong ballistic transport for a family of discrete quasiperiodic Schr\"odinger operators as a consequence of exponential dynamical localization for the dual family. The latter has been, essentially, shown by Jitomirskaya and Kr\"uger in the one-frequency setting and by Ge--You--Zhou in the multi-frequency case. In both regimes, we obtain strong convergence of $\frac{1}{T}X(T)$ to the asymptotic velocity operator $Q$, which improves recent perturbative results by Zhao and provides the strongest known form of ballistic motion. In the one-frequency setting, this approach allows to treat Diophantine frequencies non-perturbatively and also consider the weakly Liouville case.
\end{abstract}
\maketitle
\section{Introduction and main results}
In this paper, we consider the following class of multi-frequency quasiperiodic operators on $\ell^2(\mathbb Z)$:
\beq
\label{h_def}
(H(x)\psi)(n)=\psi(n+1)+\psi(n-1)+\varepsilon v(x+n\alpha)\psi(n),\quad x,\alpha\in \T^d,\quad n\in \Z.
\eeq
where
$$
n\alpha=(\{n\alpha_1\},\ldots,\{n\alpha_d\})\in \T^d.
$$
We identify $\T^d=(\R/\Z)^d$ with $[0,1)^d$ and assume that $v\in C^{\omega}(\T^d;\R)$ is a real analytic potential (considered also as a $\Z^d$-periodic function on $\R^d$). Here $\alpha=(\alpha_1,\ldots,\alpha_d)$ is the frequency vector such that $\{1,\alpha_1,\ldots,\alpha_d\}$ are independent over $\Q$. Whenever we introduce an abstract concept that does not use quasiperiodic specifics, we will use the notation $H$ for the Schr\"odinger operator.

The position operator is defined on the natural domain of definition in $\ell^2(\Z)$ by
\beq
\label{X_def}
(X\psi)(n)=n\psi(n),
\eeq
and its Heisenberg evolution can be represented as
\beq
\label{XT_def}
X(T)=e^{iTH}Xe^{-iTH}=X+\int_0^T e^{itH}Ae^{-itH}\,dt,\quad T\in \R,
\eeq
where
\beq
\label{A_def}
A\psi(n)=i(\psi(n+1)-\psi(n-1)).
\eeq
Since $A$ is bounded, \eqref{XT_def} implies that $X=X(0)$ and $X(T)$ have the same domain. We will be interested in computing the limits
\beq
\label{X_lim}
\lim\limits_{T\to +\infty}\frac{1}{T}X(T)\psi_0,
\eeq
where $\psi_0\in \dom(X)$. One can consider the limit \eqref{X_lim}, if it exists, as the ``asymptotic velocity'' of the state $\psi_0$ at infinite time. The {\it asymptotic velocity operator} is defined as
\beq
\label{Q_def}
Q=\slim\limits_{T\to +\infty}\frac{1}{T} X(T)=\slim\limits_{T\to +\infty}\frac{1}{T}\int_0^T e^{itH}Ae^{-itH}\,dt. 
\eeq
The first limit is only defined on a dense set, but it is natural to remove the term $\frac{1}{T}X(0)$ and consider only the right hand side. We say that a Schr\"odinger operator $H$ demonstrates {\it strong ballistic transport}, if the right hand side of \eqref{Q_def} converges on $\ell^2(\Z)$ and $\ker Q=\{0\}$. Strong ballistic transport immediately implies
\beq
\label{X_moment}
\|X(T)\psi_0\|\ge c(\psi_0)|T|,\quad |T|\gg 1,\quad \psi\in \dom(X).
\eeq
Ballistic motion is recognized as one of the manifestations of absolutely continuous spectrum. Originally, it was studied in the Ces\`aro averaged sense, see, for example \cite{Last} and references therein. The non-averaged lower bounds for $X(T)$ were first found in \cite{Knauf} for periodic operators in the continuum. Later, they were extended in \cite{Lukic} to the discrete Jacobi matrix case, motivated by applications to XY spin chains. Some anomalous bounds for Fibonacci type Hamiltonians were found in \cite{Lukic2}. In the quasiperiodic case, an $x$-averaged version of ballistic transport was obtained in \cite{K} by the duality method based on \cite{JK}. As a consequence, one can still obtain lower bounds on Lieb--Robinson velocity for the XY chain, but the actual ballistic transport would only be proved for a sequence of time scales. In the same year, a different approach was developed in \cite{Zhao1} in order to obtain bounds of type \eqref{X_moment} in the perturbative setting. It does not require considering a sequence of time scales, but fall short of \eqref{X_lim}. The KAM method of \cite{Zhao1} was later developed in \cite{Zhao2} to treat the one-frequency Liouvillean case, by further weakening \eqref{X_moment} to a bound on the transport exponent. The limit-periodic case was studied in \cite{Fillman} where an analogue of \eqref{X_lim} was obtained.

While \eqref{X_moment} is already a very strong condition, the convergence statement \eqref{X_lim} is more desirable, since it shows that the wavepacket takes a particular asymptotic shape at large times, assuming it is properly rescaled. One can compare this process with localization. In the quasiperiodic case, one of the results of \cite{K} is the calculation of the asymptotic velocity operator $Q(x)$, but, since it is only obtained on a sequence of time scales, one cannot exclude the possibility of large oscillations. Moreover, \cite{K} predicts a possible mechanism of convergence: after applying duality, it becomes a procedure of diagonal truncation of an operator dual to \eqref{A_def} in the basis of the eigenvectors of the dual Hamiltonian with point spectrum. The convergence of the truncation is only obtained in the dual $L^2$ direct integral space (in other words, averaged over $\theta$), which is not enough to guarantee pointwise strong convergence in the original direct integral space. A natural question arises: can we improve it? In order to obtain a pointwise bound (say, $L^{\infty}$ in the $x$ variable), can try to obtain an $\ell^1$ bound in the dual $\Z^d$ variable. Clearly, if we truncate the dual $\Z^d$ space, then $\ell^1$ bound would follow from $\ell^2$ bound, which is already obtained in \cite{K}. It turns out that the missing ingredient is a uniform $\ell^1$ bound on the tails, which is an extra property that we require from the dual model. This property follows from exponential dynamical localization and has been established in \cite{Kruger} and \cite{EDL}.

As usual, we call a frequency vector $\alpha$ {\it Diophantine} (denoted $\alpha\in \dc(c,\tau)$ for some $c,\tau>0$) if
\beq
\label{dc_def}
\dist(k\cdot\alpha,\mathbb Z)\ge c|k|^{-\tau},\quad \forall k\in \Z^d\setminus\{0\}.
\eeq
We also use the notation $\dc=\cup_{c,\tau>0}\dc(c,\tau)$. For $\alpha\in \R\setminus\Q$, denote also
$$
\beta(\alpha)=\limsup_{k\to \infty}\frac{\ln q_{k+1}}{q_k},
$$
where $\{q_k\}$ is the sequence of continued fraction approximants of $\alpha$. Note that $\alpha\in \dc$ implies $\beta(\alpha)=0$, but not vice versa. The main result of the paper is Theorem \ref{main_abstract}, which establishes strong ballistic transport as a consequence of exponential dynamical localization for the dual operator. We postpone the complete setup to Section 2, and formulate two main corollaries.
\begin{cor}
\label{cor_1}
Suppose that $d=1$, $v\in C^{\omega}(\T)$, $0<\beta(\alpha)<+\infty$. There exists $\varepsilon_0=\varepsilon_0(v,\beta)>0$ such that, for $0<\varepsilon<\varepsilon_0$, the operator $H(x)$ \eqref{h_def} has strong ballistic transport for a.e. $x\in \T$.
\end{cor}
\begin{cor}
\label{cor_2}
Suppose that $v\in C^{\omega}(\T^d)$, $\alpha\in\dc$. There exists $\varepsilon_0=\varepsilon_0(v,\alpha)>0$ such that, for $0<\varepsilon<\varepsilon_0$, the operator $H(x)$ \eqref{h_def} has strong ballistic transport for a.e. $x\in \T^d$.
\end{cor}
A version of Corollary \ref{cor_2} with \eqref{X_moment} instead of \eqref{X_lim} was obtained in \cite{Zhao1}.
\section{Preliminaries and the main result}
The proof will refine convergence bounds from \cite{K}, part of which is based on Aubry duality. Let $(\hat{v}\ast\cdot)\colon\Z^d\to \Z^d$ be the convolution operator
\beq
\label{vh_def}
(\hat{v}\ast \psi)(n)=\sum\limits_{m\in \Z^d}\hat{v}(n-m)\psi(m),
\eeq
where
$$
v(x)=\sum\limits_{m\in \Z^d} \hat{v}(m) e^{2\pi i m \cdot x}
$$
is the usual Fourier series.
The dual operator family $\wt {H}(\theta)$ is defined by
\beq
\label{hdual_def}
(\wt {H}(\theta)\psi)(m)=\varepsilon(\hat{v}\ast \psi)(m)+2\cos 2\pi (\theta+m\cdot\alpha)\psi(m),\quad \theta\in \T^1=[0,1),\quad m\in \Z^d.
\eeq
Denote the corresponding direct integral spaces (for $H$ and $\wt H$ respectively) by
$$
\mathfrak{H}:=\dirintd \ell^2(\Z)\,dx,\quad \widetilde{\mathfrak{H}}= {\int_{\T}^{\oplus}\ell^2(\Z^d)}\,d\theta.
$$
The unitary duality operator $\mathcal U\colon \mathfrak{H}\to \widetilde{\mathfrak{H}}$ is defined on functions $\Psi=\Psi(x,n)$ as
\beq
\label{u_def}
(\mathcal {U} \Psi)(\theta,m)=\wt{\Psi}(m,\theta+\alpha\cdot m),
\eeq
where $\wt\Psi$ denotes the Fourier transform over $x\in \T^d\to m\in \Z^d$ combined with the inverse Fourier transform $n\in \Z\to \theta\in \T$:
$$
\wt{\Psi}(m,\theta)=\sum\limits_{n\in \Z}\int\limits_{\T}e^{2\pi i n\theta-2\pi i m x}\Psi(n,x)\,dx.
$$
Let also
$$
\mathcal{H}:=\dirintd H(x)\,dx,\quad \wt{\mathcal{H}}:=\dirint \wt H(\theta)\,d\theta.
$$
Aubry duality (see, for example, \cite{GJLS}) can be formulated as the unitary equivalence of direct integrals
\beq
\label{Duality}
\mathcal{U} \mathcal H \mathcal U^{-1}=\wt{\mathcal H}.
\eeq 
In fact, any operator on $\mathfrak H$ has a dual counterpart, defined in a similar way. The dual version of the operator $A$ is a decomposable operator:
\beq
\label{A_dual}
(\wt A(\theta)\psi)(m)=2\sin 2\pi (m\cdot\alpha+\theta)\psi(m),\quad m\in \Z^d.
\eeq
\subsection{Strong ballistic transport in expectation} 
Denote by
\beq
\label{Q_T_def}
Q(x,T)=\frac{1}{T}\int_0^T e^{iH(x)t} Ae^{-i H(x)t}\,dt,\quad \wt Q(\theta,T)=\frac{1}{T}\int_0^T e^{i\wt H(\theta)t}\wt A(\theta)e^{-i\wt H(\theta)t}\,dt,
\eeq
and the corresponding direct integrals
$$
{\mathcal Q}(T)=\dirintd Q(x,T)\,dx,\quad \wt {\mathcal Q}(T)=\dirint \wt Q(\theta,T)\,d\theta.
$$
The following result is, essentially, established in \cite{K}.
\begin{prop}
\label{prop_K}
Suppose that the family $\wt H(\theta)$ has purely point spectrum for a.e. $\theta$. Then, for a.e. $\theta$, the following limit exists:
$$
\wt Q(\theta)=\slim\limits_{T\to +\infty}\frac{1}{T}\int_0^T e^{i\wt H(\theta)t}\wt A(\theta)e^{-i\wt H(\theta)t}\,dt.
$$
Moreover, the operator $\wt Q(\theta)$ is the diagonal part of $\wt A(\theta)$ with respect to  any orthonormal basis of eigenfunctions $\{\psi_k(\theta)\}$ of $\wt H(\theta)$:
$$
\wt Q(\theta)\psi_k(\theta)=\<\wt A(\theta)\psi_k(\theta),\psi_k(\theta)\>\psi_k(\theta),
$$
and $\ker \wt Q\neq \{0\}$ for a.e. $\theta$. As a consequence, there exist decomposable operators $\cQ$, $\wt \cQ$:
$$
\cQ=\slim\limits_{T\to +\infty}\cQ(T),\quad \wt \cQ=\slim\limits_{T\to +\infty}\wt\cQ(T);\quad \ker\cQ=\ker\wt \cQ=\{0\}.
$$
\end{prop}
\begin{remark}
In \cite{K}, Proposition \ref{prop_K} was formulated in a slightly different setting, assuming that the family $H(x)$ satisfies $L^2$ degree zero reducibility condition. However, one can check that the same proof follows through. In particular, the proof of the fact $\ker\wt \cQ=\{0\}$ is done exactly the same way as in Appendix C of  \cite{AJM}, see also Remark 5.1 in \cite{JK}.
\end{remark}
Once the convergence in $L^2$ is obtained, one can apply a diagonal procedure to establish the following fact: there is a sequence of time scales $T_k\to \infty$ as $k\to\infty$ such that, for almost every $x\in \T^d$, $Q(x,T_k)$ converges to $Q(x)$ strongly.

\subsection{Main results} Our goal is to improve the convergence of $\wt \cQ(T)$. This requires some additional information about the dual operator. Denote by $\{\delta_k\colon k\in \Z^d\}$ the standard basis in $\Z^d$.
\begin{defi}
We say that the family $\{\wt H(\theta)\}$ satisfies {\it exponential dynamical localization in expectation} if the spectra of $\wt H(\theta)$ are purely point for a.e. $\theta\in \T$, and the following bound holds with some constants $C,\gamma>0$:
\beq
\label{eq_edl}
\int\limits_{\T}\sup\limits_{t\in \R}|\<\delta_k,e^{-it\wt H(\theta)}\delta_{\ell}\>|d\theta\le Ce^{-\gamma|k-l|}.
\eeq
\end{defi}
It turns out that EDL is the missing ingredient for establishing ``true'' strong ballistic transport \eqref{X_lim}. The following is the main result of the present paper.
\begin{theorem}
\label{main_abstract}
Suppose that the family $\{\wt H(\theta)\}$ satisfies EDL. Then, for almost every $x\in\T^d$, the operator $H(x)$ has strong ballistic transport.
\end{theorem}
There are two cases in which EDL is established. The first one is the weakly Liouvillean one-frequency case. In \cite{Kruger}, it is obtained for the almost Mathieu operator. However, the proof relies on Theorem 5.1 from \cite{AJ} which has, conveniently, been obtained for the general non-local case, and the rest of the argument can be repeated verbatim. See also \cite{Germinet} for earlier application of the method and \cite{lanawenkru} for a significantly refined result for the almost Mathieu operator.
\begin{prop}
\label{prop_AJ_Kruger}
Fix $v\in C^{\omega}(\T)$ and $\beta>0$. There exists $\varepsilon_0=\varepsilon_0(v,\beta)>0$ such that the operator family
$$
(\wt {H}(\theta)\psi)(m)=\varepsilon(\hat{v}\ast \psi)(m)+2\cos 2\pi (\theta+m\cdot\alpha)\psi(m),\quad m\in \Z.
$$
satisfies EDL.
\end{prop}
Recently, a multi-dimensional analogue has been obtained in \cite{EDL}:
\begin{prop}
\label{prop_EDL}
Fix $v\in C^{\omega}(\T^d)$ and suppose that $\alpha\in \dc$. There exists $\varepsilon_0=\varepsilon_0(v,\alpha)>0$ such that the operator family
$$
(\wt {H}(\theta)\psi)(m)=\varepsilon(\hat{v}\ast \psi)(m)+2\cos 2\pi (\theta+m\cdot\alpha)\psi(m),\quad m\in \Z^d
$$
satisfies EDL.
\end{prop}
We should note that \cite{EDL} also contains a version of Proposition \ref{prop_AJ_Kruger} in the Diophantine setting, obtained by a different method from the ``reducibility'' side.
\section{Proof or Theorem \ref{main_abstract}}

Suppose that $\{\theta_j\}_{j\in \Z^d}\subset \T^1$ is some fixed sequence of phases. Denote by $L^{21}_{\dual}$ the space of functions $\Psi$ on $\T^1\times \Z^d$ with the norm
\beq
\label{L2_tilde}
\|\Psi\|_{L^{21}_{\dual}}=\l\{\int\limits_{\T^1}\l(\sum\limits_{m\in \Z^d}\l|\Psi(\theta+\theta_m;m)\r|\r)^2\,d\theta\r\}^{1/2}.
\eeq
The definition resembles the vector-valued space $L^2(\T;\ell^1(\Z^d))$. However, before calculating $\ell^1$-norm, we shear the argument of the $m$th component by $\theta_m$. Let also $P_N$ be the orthogonal projection onto $\mathrm{Span}\{\delta_n\colon n\in \Z^d, |n|\le N\}$ in $\ell^2(\Z^d)$, and $P_N^{\perp}=I-P_N$.
\begin{lemma}
\label{lemma_bound_tail}
Under the assumptions of Theorem $\ref{main_abstract}$, define $\wt Q(\theta,T)$ by \eqref{Q_T_def}. Then, the following bound holds:
\beq
\label{eq_bound_tail}
\l(\int_{\T}\|P_N^{\perp}\wt Q(\theta,T)\delta_k\|^2_{\ell^1(\Z)}\,d\theta\r)^{1/2}\le C_1 e^{-C_2 |N-|k||}.
\eeq
Moreover, the norm in the left hand side can be replaced by the norm in $L^{21}_{\dual}$ for any choice of $\{\theta_m\}$, with the same bounds.
\end{lemma}
\begin{proof}
We will prove a stronger statement: a uniform bound of \eqref{eq_bound_tail} without Ces\`aro averaging; that is, with $\wt Q(\theta,T)$ replaced by $e^{i\wt H(\theta)t}\wt A(\theta)e^{-i\wt H(\theta)t}$. Using the triangle inequality applied to the $\|\cdot\|_{L^{21}_{\dual}}$-norm of the sum 
$$
P_N^{\perp}e^{i\wt H(\theta)t}\wt A(\theta)e^{-i\wt H(\theta)t}\delta_k=
\sum_{|n|>N}\l\<\delta_n, e^{i\wt H(\theta)t}\wt A(\theta)e^{-i\wt H(\theta)t}\delta_k\r\>\delta_n,
$$
we can estimate \eqref{eq_bound_tail}:
\begin{multline}
\label{eq_many_triangles1}
\l(\int_{\T}\|P_N^{\perp}e^{i\wt H(\theta)t}\wt A(\theta)e^{-i\wt H(\theta)t}\delta_k\|_{\ell^1(\Z^d)}^2\,d\theta\r)^{1/2}
\le 
\sum_{|n|>N}\l(\int_{\T}\l|\< \delta_n,e^{i\wt H(\theta)t}\wt A(\theta)e^{-i\wt H(\theta)t}\delta_k\>\r|^2\,d\theta\r)^{1/2}\\
\le 2 \sum_{|n|>N}\l(\int_{\T}\l|\< e^{-i\wt H(\theta)t}\delta_n,\wt A(\theta)e^{-i\wt H(\theta)t}\delta_k\>\r|\,d\theta\r)^{1/2}
\\
\le 2\sum_{|n|>N}\l(\sum_{\ell\in \Z^d}\int_{\T}\l|\< e^{-i\wt H(\theta)t}\delta_n,\delta_{\ell}\>\<\wt A(\theta)\delta_{\ell},e^{-i\wt H(\theta)t}\delta_k\>\r|\,d\theta\r)^{1/2}.
\end{multline}
In the second inequality, we used the fact that $|\< \delta_n,e^{i\wt H(\theta)t}\wt A(\theta)e^{-i\wt H(\theta)t}\delta_k\>|\le 2$. Moreover, $\wt A(\theta)$ is a self-adjoint operator of multiplication by $2\sin(m\cdot\alpha+\theta)$, and therefore it can be removed from the last espression with an extra factor of $2$. We will also use the bound $|\<\delta_k,e^{-i\wt H(\theta) t}\delta_{\ell}\>|\le 1$ several times in the continued estimates:
\begin{multline*}
\\
\eqref{eq_many_triangles1}\le 4\sum_{|n|>N}\l(\sum_{\ell\in \Z^d}\int_{\T}\l|\< e^{-i\wt H(\theta)t}\delta_n,\delta_{\ell}\>\<\delta_{\ell},e^{-i\wt H(\theta)t}\delta_k\>\r|\,d\theta\r)^{1/2}
\\
 \le 4\sum_{|n|>N}\sum_{\ell\in \Z^d}\l(\int_{\T}\l|\< e^{-i\wt H(\theta)t}\delta_n,\delta_{\ell}\>\<\delta_{\ell},e^{-i\wt H(\theta)t}\delta_k\>\r|\,d\theta\r)^{1/2}
\\
\le 4\sum_{|n|>N}\sum_{\ell\in \Z^d}\l\{\int_{\T}|\<e^{-i \wt H(\theta)t}\delta_n,\delta_\ell\>|^{1/2}|\<e^{i \wt H(\theta)t}\delta_\ell,\delta_k\>|^{1/2}\,d\theta\r\}^{1/2}\\
\le 4 \sum_{|n|>N}\sum_{\ell\in \Z^d}\l\{\int_{\T}|\<\delta_\ell,e^{-i\wt H(\theta)t}\delta_n\>|\,d\theta \int_{\T}|\<\delta_\ell,e^{-i\wt H(\theta)t}\delta_k\>|\,d\theta\r\}^{1/4}\\
\le C_1 \sum_{|n|>N}\sum_{\ell\in \Z^d}e^{-C_2|n-\ell|}e^{-C_2 |\ell-k|}\le C_3 e^{-C_4 |N-|k||}.
\end{multline*}
The last inequalities follow from the EDL property. The proof for arbitrary $\theta_k$ is similar: note that we immediately use the triangle inequality, after which one can change variable in the integrand for each $n$ separately.
\end{proof}
\begin{cor}
\label{cor_conv2}
Let $\varphi_q(\theta)=e^{2\pi i \theta q}$, $\Psi=\varphi_q\delta_k\in  L^{21}_{\dual}$. Then $\wt Q\Psi\in L^{21}_{\dual}$, and
$$
\|\wt \cQ(T)\Psi-\wt \cQ\Psi\|_{L^{21}_{\dual}}\to 0\quad\textrm{ as } \quad T\to +\infty.
$$
\end{cor}
\begin{proof}
First, let us note that $\wt Q\Psi$ is well defined as an element of $\ell^2(\T;\ell^2(\Z^d))$, since $\wt \cQ$ is a bounded operator. Now, from Lemma \ref{lemma_bound_tail} (the factor $\varphi_q$ does not change \eqref{eq_bound_tail}), assuming, say, $N>2k$:
$$
\|\wt Q\Psi\|_{L^{21}_{\dual}}\le \|P_N\wt Q\Psi\|_{L^{21}_{\dual}}+\|(1-P_N)\wt Q\Psi\|_{L^{21}_{\dual}}\le N^{1/2} \|\wt Q\Psi\|_{L^2(\T;\ell^2(\Z))}+C_1(\Psi)e^{-C_2(\Psi) N}<+\infty.
$$
Similarly, one can prove second claim:
$$
\|\wt Q(T)\Psi-\wt Q\Psi\|_{\wt L^{21}}\le N^{1/2}\|\wt Q(T)\Psi-\wt Q\Psi\|_{L^2(\T;\ell^2(\Z))}+C_1(\Psi)e^{-C_2(\Psi)N}.
$$
The first term in the right hand side converges to zero, since Proposition \ref{prop_K} guarantees convergence in $L^2(\T;\ell^2(\Z))$.
\end{proof}

\noindent {\it Proof of Theorem \ref{main_abstract}.} Recall the definition:
$$
Q(x,T)=\frac{1}{T}\int_0^T e^{i H(x) t}A e^{-i H(x)t}\,dt.
$$
Since $\|A\|\le 2$, we have $\|Q(x,T)\|_{\ell^2(\Z)\to \ell^2(\Z)}\le 2$. Hence, it would be sufficient to show
$$
Q(x,T)\delta_p\to Q(x)\delta_p,\quad p\in \Z^d.
$$
for all basis elements and for almost every $x\in \T^d$. Let 
$$
w_T(x)=Q(x,T)\delta_p\in \ell^2(\Z),\quad w(x)=Q(x)\delta_p\in \ell^2(\Z),
$$
where we are assuming that $x$ belongs to the full measure set of $\T^d$ on which $Q(x)$ exists (as a fiber of $\cQ$). Let $w_T(x;n)$ denote the $n$th component of $w_T$, $n\in \Z$. Consider the Fourier transforms
$$
\hat w_T(x;\theta)=\sum\limits_{n\in\Z}e^{2\pi i n\theta}w_T(x;n),\quad \hat w(x;\theta)=\sum\limits_{n\in\Z}e^{2\pi i n\theta}w(x;n).
$$
Perform also the inverse Fourier transform in the first argument, and denote the results by $\wt w$:
$$
\wt w_T(m;\theta)=\int_{\T^d}\sum\limits_{n\in\Z}e^{-2\pi i m \cdot x}e^{2\pi i n\theta}w_T(x;n)\,dx,\quad \wt w (m;\theta)=\int_{\T^d}\sum\limits_{n\in\Z}e^{-2\pi i m \cdot x}e^{2\pi i n\theta}w(x;n)\,dx.
$$
Here $m\in \Z^d$, $\theta\in \T$. In the next computation, $\sup$ denotes $\ess\sup$. We have
\begin{multline}
\label{eq_final}	
\sup_{x\in\T^d}\|w_T(x)-w(x)\|_{\ell^2(\Z)}^2=\sup_x \sum_{n\in \Z}|w_T(x;n)-w(x;n)|^2=\sup_x \int_{\T}|\hat w_T(x;\theta)-\hat w(x;\theta)|^2\,d\theta\\
\le \int_{\T}\l(\sup_{x}|{\hat w}_T(x,\theta)-{\hat w}(x,\theta)|\r)^2\,d\theta
\le \int_{\T}\l(\sum_{m\in \Z^d}|{\wt w}_T(m,\theta)-{\wt w}(m,\theta)|\r)^2\,d\theta
\\=\int_{\T}\l(\sum_{m}\l|(\mathcal U w_T)(\theta+m\alpha,m)-(\mathcal U w)(\theta+m\alpha,m)\r|\r)^2\,d\theta\\
=\|\mathcal U w_T-\mathcal U w\|^2_{L^{21}_{\dual}},
\end{multline}
where $\mathcal U$ denotes the duality transformation \eqref{u_def} and the phases $\theta_n$ in the definition of $\wt L^{21}$ are chosen in the form $\theta_m=m\cdot\alpha$. However,  
$$
(\mathcal U w)(\theta,m)=\wt Q(\theta) e^{2\pi i p\theta}\delta_q(m),\quad (\mathcal U w_T)(\theta,m)=\wt Q(\theta,T)e^{2\pi i p \theta}\delta_q(m).
$$
Hence, the right hand side of \eqref{eq_final} converges to zero from Corollary \ref{cor_conv2}.
\section{Acknowledgements} The research was supported by the NSF DMS--1846114 grant ``CAREER: Quantum Systems with Deterministic Disorder''.

The author would like to thank Qi Zhou for pointing out an incorrect inequality in the previous version of Lemma 3.1.
\bibliography{ballistic} 
\bibliographystyle{plain}
\end{document}